\documentclass[11pt]{amsart} 
\usepackage{amssymb}
\usepackage{amsmath}
\usepackage{enumitem}
\usepackage{mathrsfs}
\usepackage[english]{babel}
\usepackage[all]{xy}
\usepackage{hyperref}
\usepackage{marginnote}
\usepackage{version}

\usepackage{stmaryrd}

\usepackage{fullpage}

\RequirePackage{mathrsfs} 

\newtheorem{theorem}{Theorem}[section]
\newtheorem{lemma}[theorem]{Lemma}
\newtheorem{proposition}[theorem]{Proposition}
\newtheorem{corollary}[theorem]{Corollary}
\newtheorem{conjecture}[theorem]{Conjecture}

\theoremstyle{definition}
\newtheorem*{ack}{Acknowledgements}
\newtheorem*{con}{Conventions}
\newtheorem{remark}[theorem]{Remark}

\newtheorem{definition}[theorem]{Definition}

\numberwithin{equation}{section} \numberwithin{figure}{section}

\DeclareMathOperator{\Pic}{Pic} 
 
\DeclareMathOperator{\Aut}{Aut}

\DeclareMathOperator{\Spec}{Spec}

\DeclareMathOperator{\Hom}{Hom}

\DeclareMathOperator{\Bs}{Bs}

\DeclareMathOperator{\Bir}{Bir}

\DeclareMathOperator{\Graph}{Graph}

\DeclareMathOperator{\Sur}{Sur}

\newcommand\ZZ{\mathbb{Z}}
\newcommand\NN{\mathbb{N}}
\newcommand\QQ{\mathbb{Q}}

\newcommand\CC{\mathbb{C}}

\def\P{\mathbb{P}}
\def\O{\mathcal{O}}
\def\bydef{:=}

\usepackage{color}

\definecolor{orange}{rgb}{1,0.5,0}

\title[Algebraic   intermediate hyperbolicities]{Algebraic  intermediate hyperbolicities}

 \author{Antoine Etesse}
\address{Antoine Etesse\\
Aix Marseille Univ\\
		CNRS, Centrale Marseille, I2M\\
		Marseille\\
		France}
\email{antoine.etesse@univ-amu.fr}

\author{Ariyan Javanpeykar}
\address{Ariyan Javanpeykar \\
Institut f\"{u}r Mathematik\\
Johannes Gutenberg-Universit\"{a}t Mainz\\
Staudingerweg 9, 55099 Mainz\\
Germany.}
\email{peykar@uni-mainz.de}

 \author{Erwan Rousseau}
\address{Erwan Rousseau \\ Institut Universitaire de France
	\& Aix Marseille Univ\\
		CNRS, Centrale Marseille, I2M\\
		Marseille\\
		France}
\email{erwan.rousseau@univ-amu.fr }
 
 \thanks{The third author was supported by the ANR project \lq\lq FOLIAGE\rq\rq{}, ANR-16-CE40-0008  and by the European Union’s Horizon 2020 research and innovation program under the Marie Sklodowska-Curie grant agreement No 75434}



\begin{document}

\maketitle

\thispagestyle{empty}
 
 \begin{abstract}    
 
We extend Lang's conjectures to the setting of intermediate hyperbolicity and prove two new results motivated by these conjectures. More precisely, we  first  extend the notion of algebraic hyperbolicity (originally introduced by Demailly) to the setting of intermediate hyperbolicity   and show that this property holds if the appropriate exterior power of the cotangent bundle is ample. Then, we prove that this intermediate algebraic hyperbolicity implies the finiteness of the group of birational automorphisms and of the set of surjective maps from a given projective variety.  Our work answers the algebraic  analogue of a question of Kobayashi on analytic hyperbolicity.
 \end{abstract}

 \section*{Introduction}  
 This paper is concerned with Lang's conjectures on hyperbolic varieties.
Lang's conjectures relate different notions of hyperbolicity for $X$ (see \cite{Lang2} and  \cite[\S12]{JBook} for a summary of his conjectures) from complex analysis to algebraic geometry and number theory. The aim of this paper is to prove several results motivated by these conjectures, and to extend Lang's conjectures on hyperbolic varieties to the more general setting of ``intermediate hyperbolicity'' (see Section \ref{section:lang}). The latter is hinted at in Lang's original paper, but not worked out anywhere in the literature in full. In fact, in \cite[page~162]{Lang2} Lang explicitly says that ``I omit the whole area of  intermediate hyperbolicity``. In this paper, we present his conjectures in the intermediate setting, and prove several new results going in their direction.

 Let us briefly explain the notion of ``intermediate hyperbolicity'' from a  complex-analytic perspective.
Recall that a complex space $X$ is \emph{Kobayashi hyperbolic} if and only if the Kobayashi pseudometric $d_X$ on $X$ is a metric (see   \cite{Kobayashi}).  In the hope of understanding the property of being Kobayashi hyperbolic for a complex space,  Eisenmann \cite{eisenman1970intrinsic} extended Kobayashi's definition and introduced  the notion of $p$-analytic hyperbolicity for all $1\leq p \leq \dim(X)$ (see \cite[Definition~1.3.(ii)]{Demailly}  for a precise definition).

We note that $1$-analytic hyperbolicity and Kobayashi hyperbolicity coincide for compact complex-analytic spaces. Furthermore, a complex space $X$ of dimension $d=\dim(X)$ is   $d$-analytically hyperbolic (as defined below) if and only if it is ``measure-hyperbolic'' (as defined in Kobayashi's book \cite{Kobayashi}). Moreover, if $X$ is $p$-analytically hyperbolic, then it is $(p+1)$-analytically hyperbolic, so that we have a string of implications 
    \[
    X \ \textrm{is} \ 1\textrm{-analytically hyperbolic} \implies   \ldots \implies X \ \textrm{is} \ (\dim X)\textrm{-analytically hyperbolic}.
    \]
    In particular, the notions of $p$-analytic hyperbolicity interpolate between the   notions of Kobayashi hyperbolicity and measure hyperbolicity.  This is why we sometimes refer to   the notion of $p$-hyperbolicity  as being an ``intermediate hyperbolicity''.
 Kobayashi-Ochiai proved that a smooth projective variety $X$ of general type (i.e., $\omega_X$ is big) is $\dim(X)$-analytically hyperbolic, and Demailly generalized their result to the setting of intermediate hyperbolicity by replacing the condition that $\omega_X$  is big by a different positivity condition on $\Lambda^p \Omega^1_X$. We state here a simpler version of his result for the sake of convenience.

\begin{theorem}[Demailly]\label{thm:dem2} Let $X$ be a smooth projective variety over $\CC$.
If  $\bigwedge^p \Omega^1_X$ is  ample, then $X$ is   $p$-analytically hyperbolic.
\end{theorem} 

Motivated by Lang's conjectures, the first aim of this paper is to prove algebraic analogues of Demailly's theorem (see Theorem \ref{thm:1} for a precise statement).  Our second aim is to investigate finiteness properties of  varieties which are hyperbolic in some intermediate sense (see Theorem \ref{thm:2}).

\subsection{An algebraic analogue of Demailly's analytic theorem on intermediate hyperbolicity}\label{results1}
  We introduce a  notion of intermediate hyperbolicity inspired by  Demailly's notion of algebraic hyperbolicity \cite{Demailly} (see also \cite{JAut, JBook, JKa, JXie, Rousseau1}). Let  us  first introduce the following convenient terminology:
 \begin{definition}[Non-degenerate maps]
A rational map \(f: Y \dashrightarrow X\) between projective varieties is said to be  \textsl{non-degenerate} if   its image \(f(Y)\) is of dimension \(\min(\dim(X), \dim(Y))\). 
 \end{definition}
Note that, if \(\dim(Y) \leq \dim(X)\), a rational map $Y\dashrightarrow X$ of projective varieties is non-degenerate if and only if it is generically finite onto its image. If \(\dim(Y)>\dim(X)\), such a rational map is non-degenerate if and only if it is dominant.

 \begin{definition}[$p$-algebraic hyperbolicity modulo $\Delta$]\label{def:alg}
 Let $X$ be a   projective variety,   let $\Delta$ be a   closed subset of $X$, and let $p$ be a positive integer. We say that $X$ is \emph{$p$-algebraically hyperbolic modulo $\Delta$} if, for every ample line bundle $L$ on $X$, there is a real number $\alpha=\alpha(X, \Delta,L)$ such that, for every smooth projective $p$-dimensional variety $Y$  with ample canonical bundle $\omega_Y$ and every non-degenerate rational map $f:Y\dashrightarrow X$ whose image is not included in \(\Delta\), the following inequality is satisfied
 \[
 (f^{*} L) \cdot (K_Y)^{p-1} \leq \alpha \cdot K_Y^p.
 \]
 
 Following Lang's terminology, we will say that $X$ is \emph{pseudo-$p$-algebraically hyperbolic} if there is a proper closed subset $\Delta \subsetneq X$ such that $X$ is $p$-algebraically hyperbolic modulo $\Delta$. Also, we will say that $X$ is \emph{$p$-algebraically hyperbolic} if $X$ is $p$-algebraically hyperbolic modulo the empty subset.
\end{definition}
Note that the pull-back by \(f\) of the line bundle  \(L\) is well-defined: since \(Y\) is smooth, the rational map \(f\) is  well-defined outside a closed subset \(F\) of codimension at least \(2\), and one  has the equality of the Picard groups \(\Pic(Y\setminus F)=\Pic(Y)\).
Note also that a projective variety $X$ is $1$-algebraically hyperbolic modulo the emptyset if and only if it is algebraically hyperbolic in Demailly's sense \cite{JKa}.   More generally, a projective variety $X$ is pseudo-$1$-algebraically hyperbolic if and only if it is pseudo-algebraically hyperbolic \cite[\S 9]{JBook}.

Observe that we restrict to varieties $Y$ with ample canonical bundle. This is analogous to the fact that  one may test the algebraic hyperbolicity of a projective variety on maps from curves of genus at least two (see the introduction of \cite{JKa} for a detailed explanation). 
Let us also emphasize on the fact that the constant \(\alpha\) appearing in the definition is \textsl{independent} of \(Y\): see Definition \ref{def:bounded} where this hypothesis is relaxed.
 
Our first result gives an algebraic analogue of Demailly's theorem (Theorem \ref{thm:dem2}). 
 To state it, we recall the definition of ampleness modulo a closed subset. We refer the reader to Lazarsfeld's books for a definition of the augmented base locus \cite{Lazzie1, Lazzie2}.

\begin{definition}\label{def:brotbek_yeng}  Let $X$ be a projective variety, let $\Delta \subset X$ be a closed subset,   let $E$ be a vector bundle on $X$, and let $p:\mathbb{P}(E^\vee)\to X $ be the natural projection. Then $E$ is \emph{ample modulo $\Delta$} if the augmented base locus $B^+(\mathcal{O}_{\mathbb{P}(E^\vee)}(1))$ is included in $p^{-1}(\Delta)$.
\end{definition}

In other words, 
$E$ is ample modulo $\Delta$, if and only if for any ample line bundle $A$ on $\mathbb{P}(E^\vee)$, there is an integer $m\geq 1$ such that the base locus of $\mathcal{O}_{\mathbb{P}(E^\vee)}(m)\otimes A^{-1}$ is included in $p^{-1}(\Delta)$.   
 
 Our first main result now reads as follows.

\begin{theorem}\label{thm:1}
Let $X$ be a smooth projective variety, and let $\Delta\subset X$ be a proper closed subset. If $\bigwedge^p\Omega^1_X$ is ample modulo $\Delta$, then $X$ is $p$-algebraically hyperbolic modulo $\Delta$.  
\end{theorem}

This   algebraic analogue of Demailly's analytic result fits in well with Lang's  ``pseudofied'' notions of hyperbolicity (see \cite{Lang2}). In fact,  Lang defines a   variety to be pseudo-hyperbolic if it is hyperbolic modulo some proper closed subset. He conjectured that pseudo-hyperbolicity (in any sense of the word ``hyperbolic'') coincides with being of general type (see \cite[page~161]{Lang2} or the more recent \cite[\S 12]{JBook}).  Demailly actually also proved a ``pseudofied'' version of his theorem (Theorem \ref{thm:dem2}). We omit a precise formulation of his result here.

The algebraic hyperbolicity of a projective variety  implies  that moduli spaces of maps from varieties are ``bounded'' (i.e. have only finitely many connected components). We refer to \cite{JKa} for precise statements.
The following definition extends the notion of boundedness introduced in \cite{JKa} (see also \cite{JBook, JXie}). 
   
 \begin{definition} [$p$ algebraically bounded modulo $\Delta$]
 \label{def:bounded}
 Let $X$ be a   projective variety,   let $\Delta$ be a   closed subset of $X$, and let $p$ be a positive integer. We say that $X$ is \emph{$p$-algebraically bounded modulo $\Delta$} if, for every ample line bundle $L$ on $X$, every smooth projective \(p\)-dimensional variety \(Y\) with \(\dim Y = p\), and every ample line bundle \(A\) on \(Y\),  there is a real number $C>0$ such that, for  every non-degenerate rational map $f:Y\dashrightarrow X$ whose image is not included in \(\Delta\), the following inequality is satisfied
 \[
 f^{*}L \cdot A^{p-1} \leq   C.
 \]

 Following Lang's terminology as before,  we will say that $X$ is \emph{pseudo-$p$-algebraically bounded} if there is a proper closed subset $\Delta \subsetneq X$ such that $X$ is $p$-algebraically bounded modulo $\Delta$. Also, we say that $X$ is \emph{$p$-algebraically bounded} if it is $p$-algebraically bounded modulo the empty subset. 
 With this terminology at hand, a projective variety is $1$-algebraically bounded over $k$ if and only if it is bounded in the sense of \cite[\S 10]{JBook}.
 \end{definition}
 
 \begin{remark} Voisin   introduced  an algebraic analogue of (analytic) measure-hyperbolicity; see \cite[Definition~2.20]{VoisinKoba}. Indeed, she defines a variety $X$ to be  \emph{algebraically measure hyperbolic}  if, for every ample line bundle $L$ on $X$, there exists a constant $A>0$ such that, for any covering family of curves $\pi: \mathcal{C} \to B$ (with generic fiber $C$ of genus $g$)  and every dominant map $\phi: \mathcal{C} \to X$ non-constant on the fibers, one has $2g-2\geq A\cdot \deg \phi_C^*L.$  Conjecturally, if $X$ is a smooth projective variety, then $X$ is $\dim X$-algebraically hyperbolic if and only if it is  algebraically measure hyperbolic (in Voisin's sense).    For example, by   \cite[Lemma~2.19]{VoisinKoba}, a variety of general type is algebraically measure-hyperbolic, and in this paper we prove the similar (a priori different) statement that a variety of general type is  $\dim X$-algebraically hyperbolic; see Theorem \ref{thm:1} in the case \(p=\dim(X)\).
\end{remark}

 The notion of intermediate boundedness is  weaker than the notion of intermediate algebraic hyperbolicity (see Proposition \ref{prop: hyp=>bound}), but we conjecture that it is equivalent (see Section \ref{section:lang}). Both notions allow for  strings of implications (see Proposition \ref{prop: string bound} and \ref{prop: string hyp} for precise statements). However, we emphasize that for intermediate algebraic boundedness this statement is only proven  under the additional assumption that the base field \(k\) is \textsl{uncountable}. 
 
 These  two intermediate notions  of pseudo-hyperbolicity turn out to be strong enough to force certain finiteness properties, as we show now.

\subsection{Kobayashi-Ochiai's finiteness theorem}\label{results2}
For our final result,  our starting point is Kobayashi-Ochiai's finiteness theorem for varieties of general type. Namely,   if $X$ is a projective variety of general type over $\CC$  and $Y$ is a projective variety, then the set of dominant rational  maps $Y\dashrightarrow X$ is finite.  As Lang's ``intermediate'' conjectures  predict that a pseudo $p$ algebraically bounded projective variety is of general type, our following result is in accordance with Lang's conjectures.  
 
\begin{theorem}\label{thm:2} Let $p\geq 1$ be an integer and let $X$ be a projective pseudo \(p\)-algebraically bounded variety over \(\CC\). Then the group of birational automorphisms $\mathrm{Bir}(X)$ is finite and for every projective variety $Y$, the set $\mathrm{Sur}(Y,X)$ of surjective morphisms $Y\to X$  is finite.
\end{theorem}


We can use Theorem \ref{thm:1} and Theorem \ref{thm:2} to reprove part of Kobayashi-Ochiai's finiteness theorem. Indeed, if $X$ is a projective variety of general type, then $X$ is $\dim(X)$ algebraically hyperbolic by Theorem \ref{thm:1} (since asking for $\bigwedge^{\dim(X)} \Omega_{X}$ to be ample modulo some proper closed subset $\Delta$ is equivalent to $X$ being of general type), hence \(\dim(X)\) algebraically bounded by Proposition \ref{prop: hyp=>bound}. In particular, by Theorem \ref{thm:2}, we deduce that $\mathrm{Bir}(X)$ is finite and $\mathrm{Sur}(Y,X)$ is finite for every projective variety $Y$. (Note that we do not obtain the finiteness of the set of dominant rational maps from $Y$ to $X$.)

Theorem \ref{thm:2} provides an ``intermediate'' version of the finiteness of $\Aut(X)$  and $\mathrm{Bir}(X)$ proven for a  pseudo-bounded projective variety $X$ in \cite{JKa, JXie}. Also, it generalizes  Matsumura's finiteness theorem that a variety of general type has only finitely many automorphisms, and provides algebraic analogues of arithmetic finiteness results proven in \cite{JAut}.

Kobayashi asked about analytic analogues of Matsumura's finiteness theorem for $\mathrm{Aut}(X)$.  In fact,  in \cite{KobayashiQuestion} Kobayashi  asked whether a $\dim(X)$ analytically hyperbolic projective variety has only finitely many automorphisms. The answer is expected to be positive, as such a variety is expected to be of general type.  For recent progress on Kobayashi's problem we refer the reader to \cite{Etesse}. Our  result (Theorem \ref{thm:2}) can be interpreted as providing a positive answer to the algebraic analogue of Kobayashi's question.

\subsection{Outline of paper}  In Section \ref{section:lang}, we  state conjectures relating several intermediate notions of hyperbolicity. We take the opportunity to   introduce several  \emph{new} arithmetic notions of intermediate hyperbolicity which will be the topic of future works. In Section \ref{section: basics}, we state and prove basic results concerning intermediate algebraic hyperbolicity and intermediate algebraic boundedness. In particular, our definitions and terminologies are (\`a posteriori) motivated and explained by Proposition \ref{prop: bound Hil}.
In Section \ref{section: ample to algebraic}, we prove our criteria ensuring pseudo \(p\)-algebraic hyperbolicity (Theorem \ref{thm:1}).
Finally, in Section \ref{section: finiteness}, we prove our finiteness results (Theorem \ref{thm:2}) for pseudo \(p\)-algebraically bounded varieties. 

\begin{ack} The second named author thanks Junyi Xie for explaining Hanamura's work on birational self-maps (Section \ref{section:hana}).
The second named author is grateful to the IHES and the University of Paris-Saclay for their hospitality. Part of this work was done during the visit of the three authors at the Freiburg Institute for Advanced Studies; they thank the Institute for providing an excellent working environment.
\end{ack}

\begin{con}
Throughout this paper, we will let \(k\) denote an algebraically closed field of characteristic zero. A \emph{variety over $k$} is a finite type separated integral (i.e. irreducible and reduced) scheme over \(k\).  Unless stated otherwise explicitly  (e.g. if we want to restrict ourselves to complex projective varieties as in Theorem \ref{thm:2}), all our varieties will be defined over \(k\).
\end{con}

 \section{Lang's intermediate algebraic and arithmetic conjectures}
 \label{section:lang}

In this section, we pursue the intermediate $p$-hyperbolicity analogues of Lang's conjectures in the algebraic and arithmetic setting (thereby leaving out the complex-analytic analogues which are discussed in \cite{Etesse}).  We will build on Lang's original conjectures \cite{Lang2} and the extensions of his conjectures summarized in \cite[\S 12]{JBook}.  

Throughout this section, let $k$ be an algebraically closed field of characteristic zero.
Given a proper scheme $X$ over $k$, we refer to \cite[\S 7]{JBook} for the definition of pseudo-Mordellicity, to \cite[\S9]{JBook} for the definition of   pseudo-algebraic hyperbolicity, and to \cite[\S 10]{JBook} for the definition of pseudo-boundedness.
The notions of $p$-algebraic hyperbolicity and $p$-algebraic boundedness are defined in the introduction (see Definition \ref{def:alg} and Definition \ref{def:bounded}). To state the general conjecture for varieties of general type, we will need one additional definition. To state this definition, we refer to \cite[\S3]{JBook} for the notion of a model.

\begin{definition}[$p$-Mordellicity modulo $\Delta$]
 Let $p\geq 0$ be an integer, let $X$ be a  proper variety over $k$, and let $\Delta\subset X$ be a closed subset.
Then, we   say that $X$ is \emph{$p$-Mordellic modulo $\Delta$ over \(k\)} if, for every finitely generated subfield $K\subset k$, every model $\mathcal{X}$ for $X$ over $K$, and every $p$-dimensional smooth projective geometrically connected variety $Y$ over $K$, the set of non-degenerate rational maps \(f:Y\dashrightarrow \mathcal{X}\) whose images are not included in \(\Delta\) is finite.
\end{definition}

\begin{definition} Let $p\geq 0$. A proper variety $X$ over $k$ is \emph{$p$-Mordellic over \(k\)} if $X$ is $p$-Mordellic modulo the empty subset over \(k\).
\end{definition}

\begin{definition} Let $p\geq 0$. A proper variety $X$ over $k$ is \emph{pseudo-$p$-Mordellic over \(k\)} if there is a proper closed subset $\Delta\subsetneq X$ such that $X$ is  $p$-Mordellic modulo $\Delta$ over \(k\).
\end{definition}

Note that $X$ is Mordellic over \(k\) (as defined in \cite[\S3]{JBook}) if and only if $X$ is $0$-Mordellic  modulo the empty subset over \(k\). Indeed, a $0$-dimensional smooth projective geometrically connected variety $Y$ over $K$ is isomorphic to $\Spec K$, so that the set  of morphisms $Y\to \mathcal{X}$ equals the set of $K$-rational points of $\mathcal{X}$.
Note also that the fact that we allow for $p=0$ in the above definition is an artifact of the arithmetic setting; finiteness of ``points'' is  a reasonable property to impose (and study) over finitely generated fields of characteristic zero.

Let $X$ be a projective variety over \(k\), and let $\Delta\subset X$ be a closed subset.   Assume that for every algebraically closed field extension $L/k$ of finite transcendence degree, the variety $X_L$  is Mordellic modulo $\Delta_L$ over $L$. Then, for every \(p\), the variety $X$ is \(p\)-Mordellic modulo $\Delta$ over \(k\). Indeed, the set of rational maps $f:Y\dashrightarrow \mathcal{X}$ with $f(Y)\not\subset \Delta$ equals the set of $K(Y)$-rational points of $\mathcal{X}\setminus \Delta$, and the latter is finite by the Mordellicity assumption on the varieties $X_L$.

For varieties of general type, we expect all notions (including the intermediate ones)  of pseudo-hyperbolicity to coincide.   The following conjecture provides a precise statement.

 \begin{conjecture}[Lang's intermediate pseudo-conjectures]\label{conj:measure}  Let $X$ be a projective variety over $k$, and let $p\geq 1$ be an integer. Then the following are equivalent.
 \begin{enumerate}
  \item The variety $X$ is of general type.
\item There is a proper closed subset $\Delta\subsetneq X$ such that every subvariety $Y\subset X$ of dimension at least $p$ with $Y\not\subset \Delta$ is of general type.
 \item The variety $X$ is pseudo-Mordellic over $k$.
 \item The variety $X$ is pseudo-$p$-Mordellic over $k$.
 \item The variety $X$ is pseudo-algebraically hyperbolic over $k$.
 \item The variety $X$ is pseudo-$p$-algebraically hyperbolic over $k$.
 \item The variety $X$ is pseudo-bounded over $k$.
 \item The variety $X$ is pseudo-$p$-algebraically bounded over $k$.
 \item The variety $X$ is $\dim(X)$-Mordellic over $k$.
 \item The variety  $X$ is $\dim(X)$-algebraically-hyperbolic over $k$.
 \item The variety $X$ is $\dim(X)$-algebraically-bounded over $k$.
 \end{enumerate}
 \end{conjecture}

 Note that (1) is independent of $p$, so that part of this conjecture is reduntant. For example, $(5)$ is equivalent to $(6)$ with $p=1$. Nonetheless, we chose to present the conjecture in this way to facilitate discussing known results in the following remark.
 
\begin{remark}[What do we know about Conjecture \ref{conj:measure} ?] The following statements hold.
\begin{enumerate}
\item Obviously, $(2)\implies (1)$.
\item   If $\dim X =1$, then Conjecture \ref{conj:measure} holds by Faltings's proof of Mordell's conjecture and the classical finiteness theorem of De Franchis-Severi for Riemann surfaces. More generally, if $X$ is a closed subvariety of an abelian variety, then conjecture \ref{conj:measure} holds  by Faltings's proof of Mordell-Lang  \cite{FaltingsLang} and the work of Ueno, Bloch-Ochiai-Kawamata \cite{Kawamata} and Yamanoi  \cite{Yamanoi} on closed subvarieties of abelian varieties. 
\item If $\dim X = 2$, then $(3)\implies (2)$, $(5)\implies (2)$, and $(7)\implies (2)$. This is explained in \cite{JBook}.
\item  By \cite{vBJK}, we have that $(5)\implies (7)$.
\item  In this paper we prove $(6)\implies (8)$ (Proposition \ref{prop: hyp=>bound}), and thus $(10)\implies (11)$. We also show that $(1)\implies (10)$ (and thus $(1)\implies (11)$). See Theorem \ref{thm:1}.  
\item Assuming $k$ is uncountable, we show that $(8)\implies (11)$; see Proposition \ref{prop: string bound}.
\item We show that, if $\Lambda^p\Omega^1_X$ is ample modulo some proper closed subset, then $X$ satisfies $(1)$, $(2)$, $(6)$, $(8)$, $(10)$, and $(11)$ (see Theorem \ref{thm:1}).
\item  We show   that, assuming $X$ satisfies $(10)$ or $(11)$, then $\mathrm{Bir}_k(X)$ is finite and $\mathrm{Sur}_k(Y,X)$ is finite for every $Y$ (see Theorem \ref{thm:2}).
\end{enumerate}
\end{remark}

Part of the above conjecture already appears in Lang's original paper. For example, the equivalence of $(1)$, and $(3)$ is stated explicitly in \cite{Lang2}. Moreover,  the equivalence of $(1)$, $(5)$, and $(7)$ is implicit in Lang's original conjectures (see \cite[\S 12]{JBook}). However, the other conjectured equivalences are new.

To conclude this section,  we push Lang's conjectures further, and state the general conjecture for ``intermediate'' exceptional loci.

\begin{conjecture}[Lang's intermediate conjectures for exceptional loci]\label{conj:measure_2} Let $X$ be a projective variety over $k$, let $\Delta\subset X$ be a closed subset, and let $p$ be a positive integer. Then the following statements are equivalent.
\begin{enumerate}
\item Every subvariety $Y$ of $X$ of dimension at least $p$ with $Y\not\subset \Delta$ is of general type.
\item The variety $X$ is Mordellic modulo $\Delta$ over $k$.
\item The variety $X$ is $p$-Mordellic modulo $\Delta$ over $k$.
\item The variety $X$ is   $p$-algebraically hyperbolic modulo $\Delta$ over $k$.
\item The variety $X$ is $p$-algebraically bounded modulo $\Delta$ over $k$.
\end{enumerate}
\end{conjecture} 

Note that  Conjecture \ref{conj:measure_2} (for all $X$ and $p$) implies Conjecture \ref{conj:measure} (for all $X$ and $p$).

\begin{remark}[What do we know about this conjecture?]  By Proposition \ref{prop: hyp=>bound}, we have that $(3)\implies (4)$. We also show that, if $\Lambda^p \Omega_X$ is ample modulo $\Delta$, then $(1), (3)$, and $(4)$ hold; see Theorem \ref{thm:1}. 
\end{remark}

 \section{Algebraic hyperbolicity and algebraic boundedness}
 \label{section: basics}
In this section, we prove a few basic facts concerning intermediate algebraic hyperbolicity and algebraic boundedness, justifying in particular our definitions. Let us first record the relationship between intermediate algebraic hyperbolicity and intermediate algebraic boundedness in the following proposition:
\begin{proposition}
\label{prop: hyp=>bound}    Let \(X\) be a projective variety and let \(\Delta\subset X\) be a proper closed subset of \(X\). 
 If \(X\) is \(p\)-algebraically hyperbolic modulo \(\Delta\), then $X$ is $p$-algebraically bounded modulo \(\Delta\).
 \end{proposition}
 \begin{proof}
 Let \(Y\) be a smooth projective variety of dimension \(p\), and consider a ramified cover \(\psi: \tilde{Y} \to Y\) with \(\tilde{Y}\) smooth and \(K_{\tilde{Y}}\) ample (see e.g. \cite{Lazzie1}). Denote \(R\) the ramification divisor so that one has the following equality of  divisors:
 \[
 K_{\tilde{Y}}
 =
 \psi^{*}K_{Y}
 +
 R.
 \]

Let \(f: Y \dashrightarrow X\) be non-degenerate rational map, whose image is not included in \(\Delta\). Fix \(L \to X\) an ample line bundle on \(X\) and \(A \to Y\) an ample line bundle on \(Y\). By hypothesis on \(X\), there exists a constant \(\alpha\) independent of \(f\) and \(Y\) such that the following inequality is satisfied
\[
(\psi \circ f)^{*}L \cdot K_{\tilde{Y}}^{p-1} 
\leq 
\alpha K_{\tilde{Y}}^{p}.
\]
Let \(m\) be an integer such that \(mK_{\tilde{Y}} - \psi^{*}A\) is ample: \(m\) does depend on \(Y\) but not on \(f\). Then one has the following inequality:
\[
(\psi \circ f)^{*}L \cdot \left(\psi^{*}A\right)^{p-1} 
\leq
\alpha m^{p-1} K_{\tilde{Y}}^{p}.
\]
Using the projection formula, one deduces that \(X\) is indeed \(p\)-algebraically bounded, as the bound on the right does not depend on the rational map \(f\).
 \end{proof}

Let us now prove that the notions of intermediate algebraic hyperbolicity and intermediate algebraic boundedness form a string of implications, starting with the case of algebraic boundedness:
 
 \begin{proposition}
 \label{prop: string bound}
 Assume that \(k\) is uncountable.
 Let \(X\) be a projective variety, let $\Delta\subset X$ be a closed subset, and let \(p \in \NN_{\geq 1}\). 
 If \(X\) is \(p\)-algebraically bounded  modulo \(\Delta\), then \(X\) is \((p+1)\)-algebraically bounded   modulo \(\Delta\).
 \end{proposition}
 \begin{proof}
 Fix \(L\) an ample line bundle on \(X\), \(Y\) a smooth projective variety of dimension \(p+1\) and \(A\) a very ample line bundle on \(Y\). 
 We argue by contradiction. Thus, suppose that there exists a sequence of non-degenerate rational maps \(f_{i}: Y \dashrightarrow X\) whose image are not included in \(\Delta\) such that $
 f_{i}^{*}L \cdot A^{p} $ tends to infinity as $i$ tends to infinity. 
As \(k\) is uncountable, one can choose a smooth ample divisor  \(H \subset Y\)  taken in the linear system \(|A|\) such that, for any \(i \in \NN_{\geq 1}\), the rational map \((f_{i})_{\vert H}: H \dashrightarrow X\) is a  well-defined non-degenerate map whose  image is not included in \(\Delta\).
 Since \(f_{i}^{*}L \cdot A^{p}=(f_{i})_{\vert H}^{*}L \cdot A_{\vert H}^{p-1}\), this contradicts the \(p\)-algebraic boundedness of \(X\).
 \end{proof}
 
Replacing ``algebraic boundedness`` by ``algebraic hyperbolicity``, we also have the following; note that the uncoutability hypothesis in the statement has been dropped.
\begin{proposition}
 \label{prop: string hyp}
Let \(X\) be a projective variety, let $\Delta\subset X$ be a closed subset, and let \(p \in \NN_{\geq 1}\).
 If \(X\) is \(p\)-algebraically hyperbolic modulo \(\Delta\), then \(X\) is \((p+1)\)-algebraically hyperbolic  modulo \(\Delta\).
\end{proposition}
\begin{proof}
Fix \(L \to X\) an ample line on \(X\), and suppose that \(X\) is \(p\)-algebraically hyperbolic modulo \(\Delta\). One must prove that there exists a constant \(\alpha\) such that for any smooth projective variety \(Y\) canonically polarized of dimension \((p+1)\) and any non-degenerate rational map \(f: Y \dashrightarrow X\) whose image is not included in \(\Delta\), one has the following inequality
\[
f^{*}L\cdot K_{Y}^{p}
\leq
\alpha K_{Y}^{p+1}.
\]
Fix \(f:Y \dashrightarrow X\) a rational map as above. By the work of Demailly and Angehrn-Siu (see e.g. \cite{Lazzie2}[10.2]), there exists a natural number \(m \in \NN\) \textsl{independent of \(Y\)} such that \(mK_{Y}\) is very ample. Take \(H \in |mK_{Y}|\) a smooth hypersurface such that \(f_{| H}: H \dashrightarrow X\) remains a non-degenerate rational map whose image is not included in \(\Delta\). By the adjunction formula, one has the following equality of linear equivalence classes of divisors
\[
K_{H}
=
(m+1)(K_{Y})_{\vert H}.
\]
Since \(X\) is \(p\)-algebraically hyperbolic modulo \(\Delta\), there exists a constant \(\beta\), independent of \(f_{\vert H}\) and \(H\), such that one has the following inequality 
\[
f_{\vert H}^{*}L \cdot K_{H}^{p-1}
\leq
\beta
K_{H}^{p}.
\]
This inequality can be rewritten as follows
\[
f^{*}L \cdot K_{Y}^{p}
\leq
\beta (m+1) K_{Y}^{p+1},
\]
with \(m\) and \(\beta\) independent of \(f\) and \(Y\), so that the proof is complete. 
\end{proof}
 
 We end this section with a classic yet fundamental application of the work of Koll\'ar and Matsusaka  \cite{KollarMat}, which justifies our definition and terminology of intermediate algebraic boundedness. Before so, let us recall the notion of Hilbert polynomials of rational maps.
Let  \(f: Y \dashrightarrow X\) be a rational map. Denote \(\Graph(f)\) the (closure of the) graph of $f$, and recall that a \emph{Hilbert polynomial of \(f\)} is obtained as follows. Fix \(A\) an ample line bundle on \(Y\) and \(L\) an ample line bundle on \(X\).  The \textsl{Hilbert polynomial of \(f\)}, computed with respect to the ample line bundles \(A \to Y\) and \(L \to X\), is by definition the Hilbert polynomial of the projective variety \(\Graph(f)\) computed with respect to the ample line bundle
\(
A \boxtimes L_{| \Graph(f)}
\).
This is the unique polynomial \(P_{f} \in \ZZ[X]\) such that for any \(m \in \NN\), \(m \gg 1\), the following equality is satisfied
\[
P_{f}(m)
=
\chi\left(\Graph(f), \left(A \boxtimes L_{| \Graph(f)}\right)^{m}\right)
=
\dim \mathrm{H}^{0}\left(\Graph(f), \left(A \boxtimes L_{| \Graph(f)}\right)^{m}\right),
\]
where the last equality follows from the ampleness of \(A \boxtimes L_{| \Graph(f)}\). 
In the case where the map \(f: Y \dashrightarrow X\) is actually a morphism, one easily sees that the Hilbert polynomial of \(f\) computed with respect to \(A \to Y\) and \(L \to X\) is also the Hilbert polynomial of \(Y\) computed with respect to the ample line bundle \(A\otimes f^{*}L\). This practical way of interpretating Hilbert polynomials of morphisms can be carried over for rational maps, provided one assumes \(Y\) to be smooth (locally factorial is actually enough):
\begin{lemma}
\label{lemma: Hilb pol}
Let \(Y, X\) be projective varieties, and suppose that \(Y\) is smooth. Let \(f: Y \dashrightarrow X\) be a rational map, and denote by \(P_{f}\) its Hilbert polynomial computed with respect to the ample line bundles \(A \to Y\) and \(L \to X\). Then for any \(m \gg 1\), one has the following equality:
\[
P_{f}(m)
=
\chi(Y, (A\otimes f^{*}L)^{m})
=
\dim \mathrm{H}^{0}(Y, (A\otimes f^{*}L)^{m}).
\]
\end{lemma}
\begin{proof}
Recall that the pull-back by \(f\) of the line bundle  \(L\) is well-defined: since \(Y\) is smooth, the rational map \(f\) is  well-defined outside a closed subset \(F\) of codimension at least \(2\), and one  has the equality of the Picard groups \(\Pic(Y\setminus F)=\Pic(Y)\).
Let \(\pi_{1}: \Graph(f) \to Y\) (resp. \(\pi_{2}: \Graph(f) \to X\)) be the first (resp. second)  projection, and observe that the pull-back \(f^{*}L\) is nef. Indeed, let \(m \in \NN_{\geq 1}\) be such that \(L^{m}\) is globally generated, and let \(y \in Y\). Pick any \(x \in f(y)\bydef \pi_{2}(\pi_{1}^{-1}(y))\), and take \(s \in \mathrm{H}^{0}(X, L^{m})\) such that \(s(x) \neq 0\). Then the section \(f^{*}s \in \mathrm{H}^{0}(Y, f^{*}L^{m})\) induced by \(f\) does not vanish at \(y\). 
Therefore, the line bundle \(A\otimes f^{*}L\) is ample.
To conclude, it is now enough to show the equality 
\[
\dim \mathrm{H}^{0}(\Graph(f), \left(A\boxtimes L\right)^{m}_{\vert \Graph(f)})
=
\dim \mathrm{H}^{0}(Y, (A\otimes f^{*}L)^{m})
\] 
for any \(m \in \NN\).
Observe that one has the following equality of (isomorphism classes of) line bundles on \(\Graph(f)\):
 \[
 \pi_{1}^{*}(A\otimes f^{*}L)=\left(A\boxtimes L\right)_{\vert \Graph(f)}.
 \]
Therefore, there is a natural injective map of vector spaces induced by \(\pi_{1}\)
\[
\begin{array}{lll}
\mathrm{H}^{0}(Y, (A\otimes f^{*}L)^{m})
&
\longrightarrow 
&
\mathrm{H}^{0}(\Graph(f), \left(A\boxtimes L\right)^{m}_{\vert \Graph(f)})
\\
s 
&
\longmapsto
&
s \circ \pi_{1}
\end{array}.
\]
In the other direction, any global section \(\tilde{s} \in \mathrm{H}^{0}(\Graph(f), \left(A\boxtimes L\right)^{m}_{\vert \Graph(f)})\) induces a global section \(s \in \mathrm{H}^{0}(Y \setminus F, (A\otimes f^{*}L)^{m})\), and Riemann's extension theorem (or an algebraic variant of it) allows to conclude.
\end{proof}
The terminology  \textit{algebraic boundedness} now comes from the following classic application of the work of Koll\'ar and Matsusaka  \cite{KollarMat} (see \cite{Lazzie2}[Thm 6.3.29]):
 \begin{proposition}
 \label{prop: bound Hil}
 Let \(X\) be a projective variety and let \(\Delta\subset X\) be a proper closed subset of \(X\). 
Suppose that \(X\) is \(p\)-algebraically bounded modulo \(\Delta\). Then for any smooth projective variety \(Y\) of dimension \(p\), the coefficients of the Hilbert polynomials of the non-degenerate rational morphisms from \(Y\) to \(X\), whose image are not included in \(\Delta\), are uniformly bounded.
 \end{proposition}
 \begin{proof}
Fix \(A \to Y\) and \(L \to X\) two ample line bundles, and let \(f: Y \dashrightarrow X\) be a non-degenerate rational map whose image is not included in \(\Delta\). By Lemma \ref{lemma: Hilb pol}, one has to bound independently of \(f\) the coefficients of the Hilbert polynomial of \(Y\) computed with respect to \(A\otimes f^{*}L\). By the work of Koll\'ar and Matsusaka \cite{KollarMat},
it suffices to bound the intersection numbers \((A+f^{*}L)^{p}\) and \((A+f^{*}L)^{p-1} \cdot K_{Y}\)  independently of \(f\).

 Recall that the hypothesis of \(p\)-algebraic boundedness gives us a constant \(C\) independent of \(f\) such that the following inequality holds
\[
f^{*}L \cdot A^{p-1}
\leq 
C.
\] 
By the numerical criterion of bigness (see e.g. \cite{Lazzie1}[Thm 2.2.15]), one deduces that for \(r>p\frac{C}{A^{p}}\), the line bundle \(rA- f^{*}L\) is big. In particular, a large enough multiple of this line bundle is effective. One now shows by induction on \(i \geq 1\) the existence of a constant \(C_{i}\) independent of \(f\) such that the following inequality holds
\[
(f^{*}L)^{i}A^{p-i} 
\leq 
C_{i}.
\]
One knows that it is satisfied for \(i=1\). Suppose that it is satisfied for \(1\leq i < p\). Since the line bundle \(f^{*}L\) is nef (see e.g. the proof of Lemma \ref{lemma: Hilb pol}), the line bundle \(f^{*}L+A\) is ample. From the effectivity of a large enough multiple of \(rA-f^{*}L\), one then deduces the following inequality
\[
(rA-f^{*}L) \cdot (f^{*}L+A)^{i} \cdot A^{p-i-1}
\geq
0.
\]
This in turn implies the following inequality
\[
(f^{*}L)^{i+1}\cdot A^{p-i-1}
\leq
(rA-f^{*}L) \cdot \left(\sum\limits_{k=0}^{i-1} \binom{k}{i} (f^{*}L)^{k}\cdot A^{p-1-k}\right) + r f^{*}L^{i}\cdot A^{p-i}.
\]
By induction hypothesis, each term appearing in the (developped) sum on the right can be bounded by a constant independent of \(f\), so that the induction is complete.

It is now clear that the above implies that one can bound \((f^{*}L+A)^{p}\) independently of \(f\). As for \((f^{*}L+A)^{p-1}\cdot K_{Y}\), one picks \(m\) such that \(mA-K_{Y}\) is ample (\(m\) is obviously independent of \(f\)), so that
\[
(f^{*}L+A)^{p-1}\cdot K_{Y}
\leq
m (f^{*}L+A)^{p-1} \cdot A
\]
and the above allows to conclude.
\end{proof}

\section{From positivity modulo $\Delta$ to  intermediate algebraic hyperbolicity}
\label{section: ample to algebraic}
  In this section, we prove the following criterion for intermediate algebraic hyperbolicity, wich must be seen as the algebraic analogue of Demailly's Theorem \ref{thm:dem2}:
\begin{theorem}
\label{thm: criteria for bounded}
Let \(X\) be a smooth projective variety, and let \(\Delta\subset X\) be a proper closed subset. If \(\bigwedge^p\Omega^1_X\) is ample modulo \(\Delta\), then \(X\) is \(p\)-algebraically hyperbolic modulo \(\Delta\).  
\end{theorem}
 
\begin{proof}
Fix \(L\) an ample line bundle on \(X\), and let \(Y\) be a  smooth projective \(p\)-dimensional variety with \(K_Y\) ample. Let also  \(f:Y\dashrightarrow  X\) be a non-degenerate rational map such that \(f(Y)\not\subset \Delta\). To prove the theorem, one needs to show that there exists a constant \(\alpha\) independent of \(f\) and \(Y\) such that the following inequality is satisfied
 \begin{equation*}
 \begin{aligned}
f^{*}L \cdot K_Y^{p-1} \leq \alpha K_{Y}^{p}.
\end{aligned}
\end{equation*}

 Let \(\pi_{X}\colon \P(\bigwedge^{p}T_{X}) \to X\)  be the natural projection onto \(X\), and let  \(\O_{\P(\bigwedge^{p}TX)}(1)\)  be
the dual of the tautological line bundle on \(\P(\bigwedge^{p}T_{X})\).
 As the rational map \(f\) is non-degenerate, it induces via its differential a rational map 
\[
\tilde{f}\colon
\xymatrix{ 
Y \simeq \P(\bigwedge^{p}T_{Y})
 \ar@{-->}[r] 
 & 
 \P\big(\bigwedge^{p} T_{X}\big),} \
(y, [v_{1}\wedge \dotsc \wedge v_{p}]) \mapsto \big(f(y), [df_{y}(v_{1})\wedge \dotsc \wedge df_{y}(v_{p})]\big).
\]
Observe that \(\tilde{f}\) is well defined outside a closed subset \(F\) of codimension at least two.
Indeed, writing $\tilde{f}$ in trivializations, one sees that the indeterminacy locus of $\tilde{f}$  comes either from the indeterminacy locus of \(f\), which has codimension at least two  as \(Y\) is smooth, or from the indeterminacy locus of a rational map from \(Y\) into the projective space \(\P\left(\bigwedge^{p} k^{\dim(X)}\right)\), which is also of codimension at least two by smoothness of \(Y\).
In particular, as 
\(
\Pic(Y) \simeq \Pic(Y\setminus F)
\)
by smoothness of \(Y\), the pull-back by $\tilde{f}$ of any line bundle on $\P(\bigwedge^{p}TX)$ is well-defined.

Observe now that the rational map \(f\) induces also the following non-trivial morphism of line bundles on \(Y \setminus F\):
\begin{displaymath}
 \overline{f} \colon
  \left|
  \begin{array}{rcl}
    (K_{Y}^{\vee})_{\vert Y\setminus F}
    & 
    \longrightarrow 
    & 
   \left(\tilde{f}^{*}\O_{ \P(\bigwedge^{p} T_{X})}(-1)\right)_{\vert Y\setminus F}
    \\
    \big(y, v_{1} \wedge \dotsc \wedge v_{p}\big) 
    & 
    \longmapsto 
    & 
   \big(y, df_{y}(v_{1}) \wedge \dotsc \wedge df_{y}(v_{p})\big).
     \\
  \end{array}
  \right..
\end{displaymath}
It extends to a non-trivial morphism of line bundles on \(Y\) by Riemann's extension theorem (or an algebraic variant of it). In particular, since 
\[
\Hom(K_{Y}^{\vee}, \tilde{f}^{*}\O_{ \P(\bigwedge^{p} T_{X})}(-1)) 
\simeq 
\mathrm{H}^{0}(Y,K_{Y}\otimes \tilde{f}^{*}\O_{ \P(\bigwedge^{p} T_{X})}(-1)),
\]
one deduces that the divisor 
\( K_{Y}\otimes \tilde{f}^{*}\O_{ \P(\bigwedge^{p} T_{X})}(-1)\)
is effective. 
On the other hand, as \(\bigwedge^{p} \Omega_{X}^1\) is ample modulo \(\Delta\), there exists  an integer \(m >0\) such that 
\[
\Bs\big(\O_{\P(\bigwedge^{p}TX)}(m)\otimes \pi_{X}^{*}L^{-1}\big) \subset \pi_{X}^{-1}(\Delta).
\] 
Note that \(m\) is independent of \(f\) and \(Y\). Since \(f(Y) \not\subset \Delta\) (by assumption), the pull-back
\[
\tilde{f}^{*}\big(\O_{\P(\bigwedge^{p}TX)}(m)\otimes \pi_{X}^{*}L^{-1}\big)
=
\tilde{f}^{*}\O_{\P(\bigwedge^{p}TX)}(m)\otimes f^{*}L^{-1}
\] 
remains effective. 
Denote \(E:=f^{*}L^{-1}\otimes K_{Y}^{\otimes m}\) so that one has the equality:
\begin{equation}
\begin{aligned}
\label{eq1}
f^{*}L+E=mK_{Y}.
\end{aligned}
\end{equation}
By the above, \(E\) is an effective divisor, as it the sum of two effective divisors.

One now concludes the proof as follows. Take \(r \in \NN\) such that \(rK_{Y}\) is very ample, and let \(H\) be a general complete intersection of \(p-1\) hypersurfaces in the linear system \(|rK_{Y}|\). Then one has the following equality 
\[
(f^{*}L)\cdot K_Y^{p-1} = \frac{1}{r^{p-1}} \deg(f^{*}L_{\vert H}).
\]
For a general \(H\), the restricted divisor \(E_{\vert H}\) remains effective, so that it follows from \eqref{eq1} that
\[
(f^{*}L)\cdot K_Y^{p-1}
\leq
 \frac{1}{r^{p-1}} 
 \deg\left((mK_{Y})_{\vert H})\right)
 =
 m K_{Y}^{p}.
 \]
This finishes the proof, as \(m\) is independent of \(f\) and \(Y\).

\end{proof}

  \section{Finiteness results for intermediate pseudo bounded and pseudo hyperbolic varieties}
  \label{section: finiteness}
In this section, we investigate the finiteness of the sets of surjective morphisms and birational automorphisms for pseudo \(p\)-algebraically bounded varieties and pseudo \(p\)-algebraically hyperbolic varieties. In the first two parts, we prove such finiteness results for \(\dim(X)\) algebraically bounded varieties. In a third and last part, we state the same finiteness results for \(p\)-algebraically bounded and \(p\)-algebraically hyperbolic varieties: even though the statements are immediate from Proposition \ref{prop: string bound} and Proposition \ref{prop: string hyp}, we choose to put them in a separate section to emphasize on the fact that the hypothesis ``\(k\) uncountable`` is \`a priori important in the case of intermediate algebraic boundedness, while it is not for intermediate algebraic hyperbolicity.

  \subsection{Surjective morphisms to \(\dim(X)\)-algebraically bounded varieties}
  \label{sse: sur dimX}
  
  Crucial to our proofs below is the (obvious) non-uniruledness of a \(\dim(X)\) algebraically bounded variety. We record this observation in the following lemma.
  
  \begin{lemma}\label{lem:non_uni}
  Let \(X\) be a projective variety. If \(X\) is \(\dim(X)\)-algebraically bounded, then \(X\) is non-uniruled. \qed
  \end{lemma}
  
 If \(X\) and \(Y\) are projective varieties, we let \(\underline{\Hom}(Y,X)\) be the scheme parametrizing morphisms from \(Y\) to \(X\). If $\mathrm{Hilb}(Y\times X)$ is the Hilbert scheme of $Y\times X$, then the morphism $\underline{\Hom}(Y,X)\to \mathrm{Hilb}(Y\times X)$ mapping a morphism $f:Y\to X$ to its graph $\Graph(f)\subset Y\times X$ is an open immersion \cite[Theorem~6.6]{Nitsure}.  
 We also let $\underline{\mathrm{Sur}}(Y,X)$ be the open and closed subscheme of $\underline{\Hom}(Y,X)$ parametrizing surjective morphisms from $Y$ to $X$. 
 
  Before stating and proving our results, recall that the scheme $\underline{\mathrm{Sur}}(Y,X)$ is of finite type  if and only if it has finitely many connected components. That is, $\underline{\mathrm{Sur}}(Y,X)$ is of finite type  if and only if, for every ample line bundle $L$ on $X$ and every ample line bundle $A$ on $Y$, there is an integer $n\geq 1$ and polynomials $\Phi_1,\ldots, \Phi_n $ in $\QQ[t]$ such that, for every surjective morphism $f:Y\to X$, the Hilbert polynomial of $f$ with respect to the ample line bundle $A\boxtimes L_{\vert \Graph(f)}$  lies in the finite set $\{\Phi_1,\ldots, \Phi_n\}$. 
 From \textsl{slight adaptations} of Proposition \ref{prop: string bound} and Proposition \ref{prop: bound Hil}, we obtain the following lemma.
  \begin{lemma}\label{lem:boundedness_of_Sur}
 Let \(X\) be a projective \(\dim X\)-algebraically bounded variety. Then for every smooth projective variety \(Y\) of dimension \(\dim(Y) \geq \dim(X)\), the scheme \(\underline{\mathrm{Sur}}(Y,X)\) is of finite type.
 \end{lemma}
 \begin{proof}
 If \(k\) is \textsl{uncountable}, this follows immediately from Proposition \ref{prop: string bound} and Proposition \ref{prop: bound Hil}.
 In a greater generality, one can proceed as in Proposition \ref{prop: string bound} and Proposition \ref{prop: bound Hil} once the following is noticed. Keep the notations of Proposition \ref{prop: string bound}, and suppose that \(p \geq \dim(X)\). The uncountability hypothesis was used to find a smooth ample divisor \(H\) in the linear system \(|A|\) such that for any \(i \in \NN\), the rational maps \((f_{i})_{\vert H}\) remain non-degenerate. 
 Suppose now that the \(f_{i}: Y \rightarrow X\) are instead surjective morphisms. It is then clear that for any smooth ample divisor \(H\), the restriction maps \((f_{i})_{\vert H}: H \to Y\) remain surjective, as for any \(y \in Y\), \(f^{-1}(\{y\})\) is positive dimensional, and thus intersects \(H\).
 \end{proof}

 To prove the rigidity of surjective morphisms, we will appeal to a theorem of Hwang-Kebekus-Peternell \cite{HKP}. Their result  relates the infinitesimal deformation space of a surjective morphism $Y\to X$ to the infinitesimal automorphisms of a suitable cover of $X$. For this reason, we investigate first the discreteness of \(\Aut(X)\), where \(\Aut(X)\) denotes the locally finite type group scheme of automorphisms of \(X\). Interestingly, to prove the rigidity of $\Aut(X)$, we will appeal to the boundedness of \(\underline{\mathrm{Sur}}(Y,X)\) (for every smooth projective variety \(Y\)) proved above.

 \begin{lemma}\label{lem:discreteness}
Let \(X\) be a projective \(\dim(X)\)-algebraically bounded  variety. Then $\mathrm{Aut}(X)$ is  zero-dimensional.
 \end{lemma}
 \begin{proof} Since $X$ is non-uniruled (Lemma \ref{lem:non_uni}),  the connected component   \(A:=\mathrm{Aut}^0(X)\) of the identity of \(\mathrm{Aut}(X)\) is an abelian variety.   For $a$ in $A$ and $x$ in $X$,  let $a\cdot x$ denote the action of $A$ on $X$. 
 Let \(\psi: \tilde{X} \to X\) be a resolution of singularities of \(X\). Consider the sequence of surjective morphisms \(f_n:A\times \tilde{X} \to X\) given by 
 \[
 f_n(a,\tilde{x}) = (na)\cdot \psi(x).
 \] 
 Since the degree of the (finite \'etale) morphism $[n]:A\times X \to A\times X$ equals $n^{2\dim A}$ and thus increases with $n$, one sees  that the Hilbert polynomials of the morphisms \(f_n:A\times \tilde{X}\to X\) are pairwise distinct. In particular, the scheme $\underline{\mathrm{Sur}}(A\times \tilde{X},X)$ is not of finite type. Since \(A\times \tilde{X}\) is smooth, this contradicts Lemma \ref{lem:boundedness_of_Sur}.
 \end{proof}
 
 We now record the basic fact that a finite surjective cover of a \(p\)-algebraically bounded projective variety remains \(p\)-algebraically bounded:
 \begin{lemma}\label{lem:covers}  
 Let $Z\to X$ be a finite surjective morphism of projective varieties. If \(X\) is \(p\)-algebraically bounded, then \(Z\) is \(p\)-algebraically bounded.
 \end{lemma}
 \begin{proof}
This is a straightforward consequence of the projection formula, and the fact that the pull-back of an ample by a finite morphism remains ample.
 \end{proof}
We now prove the desired finiteness of surjective morphisms \(Y\to X\), assuming \(X\) is \(\dim (X)\) algebraically bounded.
 
 \begin{theorem}
 \label{thm:sur} 
 If  \(X\) is  a  \(\dim(X)\)-algebraically bounded  projective variety and \(Y\) is a projective variety, then \(\underline{{\mathrm{Sur}}}(Y,X)\) is finite.
 \end{theorem}
 \begin{proof}
Observe first that one can always suppose that \(Y\) is smooth. Indeed, if one takes \(\tilde{Y} \to Y\) a resolution of singularities of \(Y\), and if one knows that \(\underline{\Sur}(\tilde{Y}, X)\) is finite, then one immediately deduces the finiteness of \(\Sur(Y,X)\). Furthermore, by Lemma \ref{lem:covers}, one can always suppose that \(X\) is normal: indeed, every surjective morphism from \(Y\) smooth to \(X\) factors uniquely through the normalization \(\tilde{X} \to X\), where the normalization map is a surjective and finite morphism.
 
By Lemma \ref{lem:boundedness_of_Sur}, the scheme \(\underline{\mathrm{Sur}}(Y,X)\) is of finite type. Thus, it suffices to show that \(\underline{\mathrm{Sur}}(Y,X)\) is zero-dimensional. To do so, let \(f:Y\to X\) be a surjective morphism. As $X$ is non-uniruled (Lemma \ref{lem:non_uni}) and normal, by the theorem of Hwang-Kebekus-Peternell \cite{HKP}, there is a finite surjective morphism $Z\to X$ and a morphism $Y\to Z$ such that $f:Y\to X$ factors as 
\[
Y\to Z\to X,
\]
 with \(\mathrm{Aut}^0(Z)\) surjecting onto the connected component of \(f\) in \(\underline{\mathrm{Sur}}(Y,X)\).  Now, as $Z\to X$ is  finite surjective and $X$ is \(\dim X\) algebraically bounded, it follows from Lemma \ref{lem:covers} that $Z$ is \(\dim X=\dim Z\) algebraically bounded. In particular, it follows from  Lemma \ref{lem:discreteness}   that   $\mathrm{Aut}^0(Z)$ is trivial. As \(\Aut^0(Z)\) surjects onto the  connected component of  \(f\) in \(\underline{\mathrm{Sur}}(Y,X)\), it follows that the latter is trivial, which finishes the proof.
 \end{proof}
 

 \subsection{Birational selfmaps of \(\dim(X)\) algebraically bounded varieties} \label{section:hana}
 Let $X$ be a projective integral variety. We define $\mathrm{Bir}(X)$ to be the subscheme of  the Hilbert scheme $\mathrm{Hilb}(X\times X)$ parametrizing, roughly speaking, closed subschemes $Z\subset X\times X$ such that $Z$ is integral and both projections $Z\to X$ are birational (see \cite[Definition~1.8]{Hanamura} for a more precise formulation).  Note that  the (abstract) group $\mathrm{Bir}(X)$ of birational selfmaps $X\dashrightarrow X$ is in bijection with the  set of  $k$-points  of the $k$-scheme $\mathrm{Bir}(X)$.

 Hanamura shows that the scheme $\mathrm{Bir}(X)$ can be endowed with the structure of a  group scheme structure, assuming that $X$ is a (terminal) minimal model (see \cite[\S3]{Hanamura}).    We will use a slight extension of his result.
 
In  fact, in \cite[\S4]{ProShramov2014} Prokhorov and Shramov show that a proper non-uniruled integral variety over $k$ has a pseudo-minimal model.
 Hanamura's main result on the scheme $\mathrm{Bir}(X)$ for minimal models $X$ is easily seen to extend to  pseudo-minimal models by following his proof closely. Indeed,  Hanamura's proof relies on the fact that on a minimal model every pseudo-automorphism is an automorphism. This property also holds for pseudo-minimal models by \cite[Corollary~4.7]{ProShramov2014}. In particular, Hanamura's work  gives the following   statement (see \cite[\S3]{Hanamura}).  
 
 \begin{theorem}[Hanamura]\label{thm:han}
 If $X$ is  a pseudo-minimal model, then 
 the   scheme  $\mathrm{Bir}(X)$ can be endowed with the structure of a group scheme (over $k$)  such  that   $\mathrm{Bir}^0(X)$ is isomorphic to $\mathrm{Aut}^0(X)$. 
 \end{theorem}
 
 We use Hanamura's structure result to prove the following finiteness result.
 
 \begin{proposition}\label{prop:bir} Let $X$ be a projective variety.
If $X$ is   $\dim(X)$-algebraically bounded, then $\mathrm{Bir}(X)$ is finite.
\end{proposition}
\begin{proof}  
As $X$ is $\dim(X)$ algebraically bounded, one knows that  \(X\) is non-uniruled (Lemma \ref{lem:non_uni}). Therefore, by the work of Prokhorov-Shramov \cite[\S4]{ProShramov2014},  the projective variety \(X\) has a pseudo-minimal model. 
Let \(Y\) be a pseudo-minimal model for \(X\). Note that \(\mathrm{Bir}(X) = \mathrm{Bir}(Y)\) (since \(X\) and \(Y\) are birational). Now,  by Theorem \ref{thm:han},   the   scheme $\mathrm{Bir}(Y)$ can be endowed with the structure of a group scheme  in such a way that   $\mathrm{Bir}^0(Y) = \mathrm{Aut}^0(Y)$. Since $X$ is  $\dim(X)$ algebraically bounded and $Y$ is birational to $X$, it follows that \(Y\) is also \(\dim(X)\) algebraically bounded. In particular, \(\Aut(Y)\) is trivial (Lemma \ref{lem:discreteness}) and thus Hanamura's group scheme \(\mathrm{Bir}(Y)\) is zero-dimensional. Finally, to conclude that \(\mathrm{Bir}(X)\) is finite, it suffices to show that \(\mathrm{Bir}(Y)\) is of finite type.  This boundedness statement is a straightforward consequence of the definition of \(\mathrm{Bir}(Y)$ and the fact that \(Y\) is $\dim(Y)$ algebraically bounded.
\end{proof}
 
 Note that the special case of Theorem \ref{thm: criteria for bounded} with \(p=\dim(X)\) combined with Theorem \ref{thm:sur} and Proposition \ref{prop:bir} gives an alternative proof of (a slightly weaker form of) Kobayachi-Ochiai finiteness theorems:
 \begin{corollary}[Kobayachi-Ochiai]
Let  $X$ be a smooth projective variety of general type. Then for any projective variety \(Y\), the set of surjective morphisms \(\Sur(Y,X)\) is finite. Furthermore, the set of bimeromorphisms \(\Bir(X)\) is also finite.
 \end{corollary}

\subsection{Finiteness results for intermediate pseudo bounded and pseudo hyperbolic varieties}
Using Proposition \ref{prop: string bound},  Theorem \ref{thm:sur} and Proposition \ref{prop:bir}, we obtain immediately the following:
 \begin{theorem}
 Assume that \(k\) is uncountable. Let   \(X\) be a projective pseudo-\(p\)-algebraically bounded variety, with \(1 \leq p \leq \dim(X)\). Then the set of bimeromorphisms \(\Bir(X)\) is finite, and for any projective variety \(Y\), the set of surjective morphisms \(\Sur(Y,X)\) is finite.
 \qed
 \end{theorem}
 In particular, this gives Theorem \ref{thm:2} as \(\CC\) is uncountable.
 Using this time Proposition \ref{prop: string bound} and Proposition \ref{prop: string hyp}, with  Theorem \ref{thm:sur} and Proposition \ref{prop:bir}, we obtain immediately the following, where the uncountability hypothesis has been dropped:
 \begin{theorem}
 Let   \(X\) be a projective pseudo-\(p\)-algebraically hyperbolic variety, with \(1 \leq p \leq \dim(X)\). Then the set of bimeromorphisms \(\Bir(X)\) is finite, and for any projective variety \(Y\), the set of surjective morphisms \(\Sur(Y,X)\) is finite.
 \qed
 \end{theorem}

  \bibliography{refsperiod}{}

\def\cprime{$'$}
\begin{thebibliography}{HKP06}

\bibitem[BJK]{vBJK}
R.~van Bommel, A.~Javanpeykar, and L.~Kamenova.
\newblock Boundedness in families of projective varieties.
\newblock {\em arXiv:1907.11225}.

\bibitem[Dem97]{Demailly}
J.-P. Demailly.
\newblock Algebraic criteria for {K}obayashi hyperbolic projective varieties
  and jet differentials.
\newblock In {\em Algebraic geometry---{S}anta {C}ruz 1995}, volume~62 of {\em
  Proc. Sympos. Pure Math.}, pages 285--360. Amer. Math. Soc., Providence, RI,
  1997.

\bibitem[Eis70]{eisenman1970intrinsic}
D.~A. Eisenman.
\newblock {\em Intrinsic measures on complex manifolds and holomorphic
  mappings}.
\newblock Number~96. American Mathematical Soc., 1970.

\bibitem[Ete]{Etesse}
A.~Etesse.
\newblock Complex-analytic intermediate hyperbolicity, and finiteness
  properties.
\newblock {\em arXiv:2011.12583}.

\bibitem[Fal94]{FaltingsLang}
G.~Faltings.
\newblock The general case of {S}. {L}ang's conjecture.
\newblock In {\em Barsotti {S}ymposium in {A}lgebraic {G}eometry ({A}bano
  {T}erme, 1991)}, volume~15 of {\em Perspect. Math.}, pages 175--182. Academic
  Press, San Diego, CA, 1994.

\bibitem[Han87]{Hanamura}
M.~Hanamura.
\newblock On the birational automorphism groups of algebraic varieties.
\newblock {\em Compositio Math.}, 63(1):123--142, 1987.

\bibitem[HKP06]{HKP}
J.-M. Hwang, S.~Kebekus, and T.~Peternell.
\newblock Holomorphic maps onto varieties of non-negative {K}odaira dimension.
\newblock {\em J. Algebraic Geom.}, 15(3):551--561, 2006.

\bibitem[Java]{JAut}
A.~Javanpeykar.
\newblock Arithmetic hyperbolicity: automorphisms and persistence.
\newblock {\em Math. Annalen, to appear. arXiv:1809.06818}.

\bibitem[Javb]{JBook}
A.~Javanpeykar.
\newblock The {L}ang-{V}ojta conjectures on projective pseudo-hyperbolic
  varieties.
\newblock {\em CRM Short Courses, Arithmetic Geometry of Logarithmic Pairs and
  Hyperbolicity of Moduli Spaces. 2020, Page 135-196}.

\bibitem[JK20]{JKa}
A.~Javanpeykar and L.~Kamenova.
\newblock Demailly's notion of algebraic hyperbolicity: geometricity,
  boundedness, moduli of maps.
\newblock {\em Math. Z.}, 296(3-4):1645--1672, 2020.

\bibitem[JX]{JXie}
A.~Javanpeykar and J.~Xie.
\newblock Finiteness properties of pseudo-hyperbolic varieties.
\newblock {\em IMRN, to appear. arXiv:1909.12187}.

\bibitem[Kaw80]{Kawamata}
Y.~Kawamata.
\newblock On {B}loch's conjecture.
\newblock {\em Invent. Math.}, 57(1):97--100, 1980.

\bibitem[KM83]{KollarMat}
J.~Koll{\'a}r and T.~Matsusaka.
\newblock Riemann-{R}och type inequalities.
\newblock {\em Amer. J. Math.}, 105(1):229--252, 1983.

\bibitem[Kob93]{KobayashiQuestion}
S.~Kobayashi.
\newblock Some problems in hyperbolic complex analysis.
\newblock In {\em Complex geometry ({O}saka, 1990)}, volume 143 of {\em Lecture
  Notes in Pure and Appl. Math.}, pages 113--120. Dekker, New York, 1993.

\bibitem[Kob98]{Kobayashi}
S.~Kobayashi.
\newblock {\em Hyperbolic complex spaces}, volume 318 of {\em Grundlehren der
  Mathematischen Wissenschaften [Fundamental Principles of Mathematical
  Sciences]}.
\newblock Springer-Verlag, Berlin, 1998.

\bibitem[Lan86]{Lang2}
S.~Lang.
\newblock Hyperbolic and {D}iophantine analysis.
\newblock {\em Bull. Amer. Math. Soc. (N.S.)}, 14(2):159--205, 1986.

\bibitem[Laz04a]{Lazzie1}
R.~Lazarsfeld.
\newblock {\em Positivity in algebraic geometry. {I}}, volume~48 of {\em
  Ergebnisse der Mathematik und ihrer Grenzgebiete. 3. Folge. A Series of
  Modern Surveys in Mathematics [Results in Mathematics and Related Areas. 3rd
  Series. A Series of Modern Surveys in Mathematics]}.
\newblock Springer-Verlag, Berlin, 2004.
\newblock Classical setting: line bundles and linear series.

\bibitem[Laz04b]{Lazzie2}
R.~Lazarsfeld.
\newblock {\em Positivity in algebraic geometry. {II}}, volume~49 of {\em
  Ergebnisse der Mathematik und ihrer Grenzgebiete. 3. Folge. A Series of
  Modern Surveys in Mathematics [Results in Mathematics and Related Areas. 3rd
  Series. A Series of Modern Surveys in Mathematics]}.
\newblock Springer-Verlag, Berlin, 2004.
\newblock Positivity for vector bundles, and multiplier ideals.

\bibitem[Nit05]{Nitsure}
N.~Nitsure.
\newblock Construction of {H}ilbert and {Q}uot schemes.
\newblock In {\em Fundamental algebraic geometry}, volume 123 of {\em Math.
  Surveys Monogr.}, pages 105--137. Amer. Math. Soc., Providence, RI, 2005.

\bibitem[PS14]{ProShramov2014}
Yuri Prokhorov and Constantin Shramov.
\newblock Jordan property for groups of birational selfmaps.
\newblock {\em Compos. Math.}, 150(12):2054--2072, 2014.

\bibitem[Rou10]{Rousseau1}
Erwan Rousseau.
\newblock Hyperbolicity of geometric orbifolds.
\newblock {\em Trans. Amer. Math. Soc.}, 362(7):3799--3826, 2010.

\bibitem[Voi03]{VoisinKoba}
Claire Voisin.
\newblock On some problems of {K}obayashi and {L}ang; algebraic approaches.
\newblock In {\em Current developments in mathematics, 2003}, pages 53--125.
  Int. Press, Somerville, MA, 2003.

\bibitem[Yam15]{Yamanoi}
K.~Yamanoi.
\newblock Holomorphic curves in algebraic varieties of maximal {A}lbanese
  dimension.
\newblock {\em Internat. J. Math.}, 26(6):1541006, 45, 2015.

\end{thebibliography}
\bibliographystyle{alpha}

\end{document}